\documentclass[a4paper]{amsart}
\usepackage[utf8]{inputenc}

\usepackage{amsmath, amssymb, amsfonts, amsthm, tensor, dsfont}
\usepackage[foot]{amsaddr}
\usepackage[svgnames]{xcolor}
\usepackage[style = alphabetic]{biblatex}
\usepackage[all]{xy}
\usepackage{hyperref}
\usepackage[hyphenbreaks]{breakurl}
\usepackage{subfiles}

\hypersetup{colorlinks = true, citecolor = MediumBlue, linkcolor = Crimson}

\newtheorem{thm}{Theorem}[section]
\newtheorem{lem}[thm]{Lemma}
\newtheorem{prop}[thm]{Proposition}
\newtheorem{cor}[thm]{Corollary}

\theoremstyle{definition}

\newtheorem{mydef}[thm]{Definition}

\newcommand{\R}{\mathbb{R}}

\newcommand{\N}{\mathbb{N}}

\newcommand{\E}{\mathcal{E}}

\newcommand{\al}{\alpha}
\newcommand{\bb}{\beta}

\newcommand{\om}{\omega}
\newcommand{\ee}{\varepsilon}

\newcommand{\dd}{\partial}

\newcommand{\PSH}{\text{PSH}}
\newcommand{\CAP}{\text{Cap}}

\begin{document}

\title{A complete metric topology on relative low energy spaces}
\author{Prakhar Gupta }
\address{Department of Mathematics, University of Maryland, College Park, Maryland, USA}
\email{pgupta8@umd.edu}
\begin{abstract}
    In this paper, we show that the low energy spaces in the prescribed singularity case $\E_{\psi}(X,\theta,\phi)$ have a natural topology which is completely metrizable. This topology is stronger than convergence in capacity. 
\end{abstract}



\maketitle

\tableofcontents

\section{Introduction} \label{sec: introduction}

Let $(X,\om)$ be a compact K\"ahler manifold. By the $dd^{c}$-lemma, any K\"ahler metric cohomologous to $\om$ is of the form $\om_{u} := \om + dd^{c}u$. This led to studying the space
\[
\mathcal{H} = \{ u\in C^{\infty}(X) : \om_{u} := \om + dd^{c}u > 0\}
\]
of smooth functions on $X$ to find a canonical metric in the same cohomology class as $\om$. Mabuchi \cite{mabuchiriemannianstructure}, Semmes \cite{semmesriemannianmetric}, and Donaldson \cite{donaldsonriemannianmetric} independently found a Riemannian structure on $\mathcal{H}$ given by 
\[
\langle \phi, \psi \rangle_{u} := \frac{1}{\text{Vol(X)}} \int_{X} \phi\psi \om^{n}_{u} 
\]
for $u \in \mathcal{H}$ and $\phi,\psi \in T_{u}\mathcal{H} = C^{\infty}(X)$.  Later, Chen \cite{chenspaceofkahlermetrics} showed that this Riemannian structure makes $\mathcal{H}$ a metric space $(\mathcal{H},d_{2})$. 

\sloppy Darvas \cite{darvascompletionofspaceofkahlerpotentials} showed that the completion $\overline{(\mathcal{H}, d_{2})}$ can be identified with $ (\E^{2}(X,\om),d_{2})$, where  $\E^{2}(X,\om)$ is the space of finite energy introduced by Guedj-Zeriahi \cite{GZweightedmongeampereenergy}. Upon introducing the Finsler type structure on $\mathcal{H}$, Darvas \cite{darvasgeoemtryoffiniteenergy} introduced metrics $d_{p}$ on $\mathcal{H}$ for $p \geq 1$ and showed that their completions are $(\E^{p}(X,\om),d_{p})$. The metric $(\mathcal{H}, d_{1})$ and its completion $(\E^{1}(X,\om),d_{1})$ were useful in studying special metrics on K\"ahler manifolds  \cite{darvasrubinsteinpropernessconjecture}, \cite{bermandarvaskenregyandkstability}, \cite{chenchengcsckI},\cite{chenchengcscKII}. 

This led to further attempts to find natural complete metrics on the various subspaces of $\PSH(X,\theta)$ for varying $\theta$. Recall that, for a smooth closed $(1,1)$-form $\theta$ we say that an integrable function $u : X \to \R\cup\{-\infty\}$ is $\theta$-psh if locally $u$ can be written as a sum of a smooth function and a plurisubharmonic function and $\theta + dd^{c}u \geq 0$ in the sense of currents. The following list of works illustrates the interest in finding natural metrics on these spaces. 
\begin{enumerate}
    \item Darvas-Di Nezza-Lu \cite{darvas2021l1} found that the space $\E^{1}(X,\theta)$ has a complete metric for $\theta$ merely big. 
    \item Using approximation by K\"ahler classes Di Nezza-Lu found that the spaces $\E^{p}(X,\beta)$ for $p\geq 1$ have a complete metric for $\{\beta\}$ a nef and big class in $H^{1,1}(X,\R)$.
    \item In \cite{TrusianiL1metric}, Trusiani showed that the space $\E^{1}(X,\om, \phi)$ has a complete  metric topology where $\om$ is a K\"ahler form and $\phi \in \PSH(X,\om)$ is a model potential.  The space $\E^{1}(X,\om, \phi)$ consists of potentials more singular than $\phi$ and having finite energy relative to $\phi$. See Section \ref{sec: preliminaries} for relevant definitions and results about model potentials and spaces with prescribed singularities.
    \item Xia \cite{Xia2019MabuchiGO} extended this result to show that the spaces $\E^{p}(X,\theta,\phi)$ have complete metric space structure for $\theta$ having big cohomology class and $\phi \in \PSH(X,\theta)$ a model potential.  
    \item Most recently, Darvas \cite{darvas2021mabuchi} showed that the space $\E_{\psi}(X,\om)$ has a natural complete metric, when $\om$ is K\"ahler and $\psi$ is a low energy weight as introduced by Guedj-Zeriahi \cite{GZweightedmongeampereenergy}. In the process, Darvas used the geodesics on $\mathcal{H}$ introduced by \cite{mabuchiriemannianstructure}, \cite{semmesriemannianmetric}, and \cite{donaldsonriemannianmetric}. 
\end{enumerate}

Note that only \cite{darvas2021mabuchi} deals with the low energy weights. Working with the low energy space is desirable because 
\[
\E(X,\om) = \bigcup_{\psi} \E_{\psi}(X,\om)
\]
where the union is over all low energy weights (see \cite[Proposition 2.2]{GZweightedmongeampereenergy}). However, all the finite energy spaces $\E^{p}(X,\om)$ are contained in $  \E^{1}(X,\om) \subsetneq \E(X,\om)$. Another method of measuring distance between potentials $u,v \in \E(X,\om)$ is proposed by Lempert \cite{lempertspaceofplurisubharmonicfunctions} where he measures the distance $\rho(u,v)$ by a function $\rho(u,v) : (0,V) \to \R$ where $V = \int_{X}\om^{n}$ is the volume of $(X,\om)$.

In \cite{darvas2021mabuchi}, the author noted that the metric $d_{\psi}$ on the space $\E_{\psi}(X,\om)$ satisfies 
\[
d_{\psi}(u,v) \leq \int_{X} \psi(u-v)(\om^{n}_{u} + \om^{n}_{v}) \leq 2^{2n+5} d_{\psi}(u,v)
\]
for any $u,v \in \E_{\psi}(X,\om)$. In \cite{darvas2021mabuchi}, the author asked if the central expression in the above equation can be shown to satisfy a quasi-triangle inequality, without constructing the metric $d_{\psi}$ using geodesics in $\mathcal{H}$, and thus show that the spaces $\E_{\psi}(X,\theta)$ have a quasi-metric space structure. This is the question we answer in this paper, in a more general context of low energy spaces in the prescribed singularity setting. Note that in the big case and in the prescribed singularity case, there are no $C^{1,1}$ geodesics, and hence the methods of \cite{darvas2021mabuchi} do not work. 

\begin{thm}\label{thm: complete metric space structure on prescribed singularity case}
Let $\theta$ be a closed smooth $(1,1)$-form whose cohomology class is big. Let $\phi \in \PSH(X,\theta)$ be a model potential such that $\int_{X}\theta^{n}_{\phi} > 0$. Then for any $u,v \in \E_{\psi}(X,\theta,\phi)$,
\begin{equation} \label{eq: proposed quasi-metric}
    I_{\psi}(u,v) := \int_{X} \psi(u-v)(\theta^{n}_{u} + \theta^{n}_{v})
\end{equation}
is a quasi-metric. Moreover, the topology induced on $\E_{\psi}(X,\theta,\phi)$ by $I_{\psi}$ is completely metrizable. 
\end{thm}

Here we say a few words about the prescribed singularity setting. Let $\phi \in \PSH(X,\theta)$ with $\int_{X}\theta^{n}_{\phi} > 0$. We denote by $\PSH(X,\theta,\phi)$ the set of $\theta$-psh functions $u$ that are more singular than $\phi$, meaning $u \leq \phi + C$ for some constant $C$. In particular, $\PSH(X,\theta) = \PSH(X,\theta, V_{\theta})$. The set of relatively full mass potentials is given by
\[
\E(X,\theta,\phi) = \{ u \in \PSH(X,\theta,\phi) : \int_{X} \theta^{n}_{u} = \int_{X} \theta^{n}_{\phi}\}
\]
and the set of relatively finite energy is given by 
\[
\E_{\psi}(X,\theta,\phi) = \left\{ u \in \E(X,\theta, \phi) : \int_{X} \psi(u-\phi) \theta^{n}_{u} < \infty\right\}.
\]
See Section~\ref{subsec: prescribed singularity setting} to learn more about potentials in prescribed singularity setting.

After this we study some properties of the new topology on $\E_{\psi}(X,\theta,\phi)$. The usual topology on $\E_{\psi}(X,\theta,\phi)$ given by $L^{1}(\om^{n})$ is not satisfactory for the purposes of studying Monge-Amp\`ere equation primarily because for $u_{k}, u \in \E_{\psi}(X,\theta,\phi)$ such that $u_{k} \to u$ in $L^{1}(\om^{n})$ does not imply that the non-pluripolar Monge-Amp\`ere measures satisfy $\theta^{n}_{u_{k}} \to \theta^{n}_{u}$. Hence the following result shows that the new topology is stronger and more natural. 

\begin{thm} \label{thm: convergence of measures}
If $u_{k},u \in \E_{\psi}(X,\theta,\phi)$ such that $I_{\psi}(u_{k},u) \to 0$ as $k \to \infty$ then $u_{k} \to u$ in capacity and hence $\theta^{n}_{u_{k}} \to \theta^{n}_{u}$ weakly as measures. In particular, $\int_{X} |u_{k} - u|\om^{n} \to 0$ as $k \to \infty$ as well. 
\end{thm}
This result is new  in the K\"ahler case when $\psi(t)$ grows slower than $|t|$. When $\psi(t) = |t|$ and $\phi = V_{\theta}$ the model potential with minimal singularity, in \cite{BBGZ13} the authors show that a closely related functional, 
\[
\tilde{I}(u,v) = \int_{X} (u-v)(\theta^{n}_{u} - \theta^{n}_{V})
\]
satisfy the same properties as in Theorem~\ref{thm: convergence of measures}. Moreover, in \cite{Trusianistrongtopologyofpshfunctions} the author shows the same results using $\tilde{I}$ for $\psi(t) = |t|$, $\theta$ a K\"ahler class and $\phi$ any model prescribed singularity.

At the end we give an application to the K\"ahler-Ricci flow. Guedj-Zeriahi \cite{Guedj2013RegularizingPO} and \cite{dinezzalukahlerricciflow} showed that given a potential $u \in \PSH(X,\om)$  such that $u$ has zero Lelong numbers we have smooth function $u_{t}$ for $t > 0$ such that 
\begin{equation} \label{eq: potential version of kahler ricci flow}
\frac{\dd u_{t}}{\dd t} = \log \left[ \frac{ (\om + dd^{c}u_{t})^{n}}{\om^{n}} \right], u_{t} \to u \text{ in } L^{1}(\om^{n}) \text{ as } t \to 0.
\end{equation}

If $\om_{t}: = \om + dd^{c}u_{t}$, then the above equation is equivalent to
\[
\frac{\dd \om_{t}}{\dd t} = - \text{Ric}(\om_{t}) + \text{Ric}(\om).
\]
Moreover, these functions satisfy $u_{t} \to u$ in capacity as $t\to 0$. In case $u \in \E_{\psi}(X,\om)$ we show a stronger convergence $u_{t} \to u$ as $t \to 0$ in the following theorem. 

\begin{thm} \label{thm: kahler-ricci flow}
If $u \in \E_{\psi}(X,\om)$ and $u_{t}$ satisfy~\eqref{eq: potential version of kahler ricci flow} then $I_{\psi}(u_{t}, u ) \to 0$ and thus by Theorem~\ref{thm: convergence of measures}, we recover that $u_{t} \to u$ in capacity and $\om^{n}_{u_{t}} \to \om^{n}_{u}$ weakly. 
\end{thm}

Here we point out that Theorem~\ref{thm: kahler-ricci flow} shows that any non-pluripolar measure can be approximated by measures with smooth density using K\"ahler-Ricci flow.

\subsection*{Organization}
In Section \ref{sec: preliminaries}, we setup the notation and mention all the relevant results required for the theorems we prove. In Section \ref{sec: quasi-metric space}, we show that $I_{\psi}$ is a quasi-metric on $\E_{\psi}(X,\theta,\phi)$. In Section \ref{sec: Completeness}, we show that the induced topology on $\E_{\psi}(X,\theta,\phi)$ is completely metrizable thereby completing the proof of Theorem \ref{thm: complete metric space structure on prescribed singularity case}. In Section \ref{sec: properties of the new topology}, we discuss some relevant properties of the new topology and prove Theorem~\ref{thm: convergence of measures}. In Section \ref{sec: Kahler Ricci flow}, we discuss an application to K\"ahler Ricci flow and prove Theorem~\ref{thm: kahler-ricci flow}.

\subsection*{Acknowledgements }
Special thanks to my thesis advisor Tam\'as Darvas for proposing this problem to me and for his continuous guidance and encouragement. 

I am grateful to  Duc-Viet Vu for pointing out that the methods here work in the general case of prescribed singularity types. I also learned from him that he independently obtained Theorem~\ref{thm : quasi-triangle inequality in the prescribed singularity case} using different methods. 

I thank Chinh H. Lu for explaining arguments from \cite{Guedj2019PlurisubharmonicEA}. 

I thank Antonio Trusiani for pointing me to results in literature which helped in improving the presentation of the paper. I also thank  the anonymous referee for careful reading that helped me improve the exposition of the paper. 

Research partially supported by NSF CAREER grant DMS-1846942.

\section{Preliminaries} \label{sec: preliminaries}
%Here I have to write all the background needed to read the paper and state the theorems I'm going to use in the paper which were proven earlier by other people. 

In this section, we fix the notations and give relevant definitions and results. 

We work on a fixed K\"ahler manifold $(X,\om)$. Let $\theta$ be a closed smooth $(1,1)$-form on $X$. An integrable function $u : X \to \R\cup\{-\infty\}$ is a $\theta$-psh if locally $u$ can be written as a sum of a smooth function and a plurisubharmonic function and $\theta + dd^{c}u \geq 0$ in the sense of currents. We denote by $\PSH(X,\theta)$ the set of all $\theta$-psh functions. Let $\al \in H^{1,1}(X,\R)$ be the cohomology class represented by $\theta$. We say $\al$ is big if there exists $\ee > 0$ and $u \in \PSH(X,\theta - \ee\om)$. See \cite{Boucksom2008MongeAmpreEI} to learn more about pluripotential theory on compact K\"ahler manifolds in big cohomology classes. In particular, op. cit. describes how to define the non-pluripolar Monge-Amp\`ere measure $\theta^{n}_{u}$ for any $u \in \PSH(X,\theta)$.

\subsection{Prescribed Singularity setting} \label{subsec: prescribed singularity setting}
In \cite{Darvasmonotonicity} and \cite{darvaslogconcavity} the authors developed the theory of pluripotential theory in prescribed singularity setting. For two potentials $u, v \in \PSH(X,\theta)$, we say that $u$ is \emph{more singular than} $v$ if $u \leq v + C$ for some constant $C$. We denote by $[u] = [v]$ the fact that $u$ and $v$ have the same singularity type. Given a potential $\phi \in \PSH(X,\theta)$, we denote by $\PSH(X,\theta,\phi)$ the set of all potentials $v \in \PSH(X,\theta)$ such that $v$ is more singular than $\phi$. To solve the Monge-Amp\`ere equation with prescribed singularity, Darvas-Di Nezza-Lu \cite{Darvasmonotonicity} defined the space of relatively full mass potentials as 
\[
\E(X,\theta, \phi) := \left\{ u \in \PSH(X,\theta,\phi) : \int_{X} \theta^{n}_{u} = \int_{X}\theta^{n}_{\phi}]\right\}. 
\]
Next, we want to define some subspaces of $\E(X,\theta,\phi)$ which consists of potentials having relatively finite `$\psi$-energy' for some weight $\psi$. By a weight $\psi$, we mean a function $\psi: \R \to \R$ such that $\psi$ is even, continuous, $\psi(0) = 0$, $\psi(\pm \infty) = \infty$, and on $(0,\infty)$ $\psi$ is smooth and increasing. We say $\psi$ is \emph{low energy} if $\psi$ is concave on $(0,\infty)$ and $\psi$ is \emph{high energy} if $\psi$ is convex on $(0,\infty)$. We denote by $\mathcal{W}^{-}$ the set of all low energy weights. For example, the weight $\psi(t) = |t|^{p}$ is high energy for $p \geq 1$ and low energy for $0 < p \leq 1$. For each weight $\psi \in \mathcal{W}^{-}$, we define the space of potentials with finite $\psi$-energy relative to $\phi$ as 
\[
\E_{\psi}(X,\theta,\phi) = \left\{ u \in \E(X,\theta,\phi) : \int_{X} \psi(u-\phi)\theta^{n}_{u} < \infty \right\}.
\]
As in the K\"ahler case, in the prescribed singularity case we also have
\[
\E(X,\theta,\phi) = \bigcup_{\psi \in \mathcal{W}^{-}} \E_{\psi}(X,\theta,\phi).
\]

We list a few results about the spaces $\E_{\psi}(X,\theta,\phi)$ that we use frequently in the rest of the paper.

\begin{lem}[Comparison Principle] \label{lem: comparison principle}
 \emph{\cite[Corollary 3.6]{Darvasmonotonicity}} Let $\phi \in \PSH(X,\theta)$ and $u,v \in \E(X,\theta,\phi)$. Then 
\[
\int_{\{ u< v\}} \theta^{n}_{v} \leq \int_{\{u<v\}} \theta^{n}_{u}.
\]
\end{lem}

\begin{lem}[The Fundamental inequality] \label{lem: the fundamental inequality}
\emph{\cite[Proposition 3.3]{thai2021complex}} If $u,v \in \PSH(X,\theta,\phi)$ are such that $u\leq v \leq \phi$ then 
\[
\int_{X} \psi(v-\phi)\theta^{n}_{v} \leq 2^{n+1}\int_{X} \psi(u-\phi)\theta^{n}_{u}.
\]
Thus for any $u \in \E_{\psi}(X,\theta,\phi)$ and $v \in \PSH(X,\theta,\phi)$ such that $u \leq v$ we have $v \in \E_{\psi}(X,\theta,\phi)$. 
\end{lem}

\begin{proof}
We give a simplified proof based on the argument of \cite[Lemma 2.4]{darvaslogconcavity}. 

\begin{align*}
    \int_{X}\psi(v-\phi) \theta^{n}_{v} &= \int_{0}^{\infty} \psi'(t)\theta^{n}_{v}(v < \phi - t)dt \\
    &= 2\int_{0}^{\infty}\psi'(2t) \theta^{n}_{v}(v < \phi - 2t) dt.
    \intertext{Now observe that $\{ v < \phi - 2t\} \subset \{ u < \frac{v+\phi}{2} - t\} \subset \{u < \phi - t\}$, thus using Lemma~\ref{lem: comparison principle} and the fact that $\theta^{n}_{v} \leq 2^{n} \theta^{n}_{(v+\phi)/2}$}
    &\leq 2\int_{0}^{\infty} \psi'(2t) \theta^{n}_{v} \left( u < \frac{v+\phi}{2} - t\right) dt \\
    &\leq 2^{n+1} \int_{0}^{\infty} \psi'(2t) \theta^{n}_{(v+\phi)/2} \left( u < \frac{v+\phi}{2} - t\right) dt \\
    &\leq 2^{n+1} \int_{0}^{\infty} \psi'(2t) \theta^{n}_{u} (u < \phi - t) dt.
    \intertext{Since $\psi$ is concave, we get $\psi'(2t) \leq \psi'(t)$.}
    &\leq 2^{n+1}\int_{0}^{\infty} \psi'(t) \theta^{n}_{u}(u<\phi - t)dt \\
    &= 2^{n+1} \int_{X} \psi(u-\phi) \theta^{n}_{u}.
\end{align*}
\end{proof}

\begin{lem}[Integrability] \label{lem: integrability}
Given $u,v \in \E_{\psi}(X,\theta, \phi)$ we have
\[
\int_{X}\psi(u-\phi)\theta^{n}_{v} < +\infty.
\]
In particular, if $u,v \leq \phi$, then 
\[
\int_{X}\psi(u-\phi)\theta^{n}_{v} \leq 2\int_{X}\psi(u-\phi)\theta^{n}_{u} + 2\int_{X}\psi(v-\phi)\theta^{n}_{v}.
\]
\end{lem}

\begin{proof}
The proof again generalizes the proof of \cite[Lemma 2.5]{darvaslogconcavity} and \cite[Proposition 2.5]{GZweightedmongeampereenergy}.
\begin{align*}
    \int_{X} \psi(u - \phi) \theta^{n}_{v} &= \int_{0}^{\infty} \psi'(t) \theta^{n}_{v} \{t < \phi - u\} dt \\
    &= 2 \int_{X}\psi'(2t) \theta^{n}_{v} \{ 2t < \phi - u\} dt \\
     &\leq 2\int_{X} \psi'(t) \theta^{n}_{v} \{u < \phi - 2t\} dt. 
      \intertext{ Notice that $\{ u < \phi - 2t\} \subset \{ u < -t + v\} \cup \{ v< \phi - t\}$. This gives }
    \theta^{n}_{v}\{ u < \phi - 2t\} &\leq \theta^{n}_{v} \{ u < -t + v\} + \theta^{n}_{v}\{ v < \phi - t\}
    \intertext{and Lemma~\ref{lem: comparison principle} gives }
    &\leq \theta^{n}_{u}\{u<-t+v\} + \theta^{n}_{v}\{ v <\phi - t\}. 
    \intertext{Notice that $\{ u < -t + v\} = \{ u - \phi < -t + v- \phi\}$. Since $v - \phi \leq 0$, we get $\{u < -t + v\} \subset \{u - \phi < -t\}.$ This gives us}
    &\leq \theta^{n}_{u}\{u-\phi < -t\} + \theta^{n}_{v}\{ v-\phi < -t\}. 
    \intertext{Using this we get}
    &\leq 2 \int_{X}\psi'(t) (\theta^{n}_{u} \{ u - \phi < - t\} + \theta^{n}_{v}\{ v - \phi < -t\}) \\
    &= 2\int_{X} \psi(u-\phi) \theta^{n}_{u} + 2\int_{X}\psi(v-\phi)\theta^{n}_{v}. 
\end{align*}

Now assume $u,v \in \E_{\psi}(X,\theta,\phi)$ then for some $C$, we have $u,v \leq \phi + C$. Let $\tilde{u} = u - C$ and $\tilde{v} = v - C$, so $\tilde{u},\tilde{v} \leq \phi$. Then, 
\begin{align*}
    \int_{X} \psi(u - \phi) \theta^{n}_{v} &= \int_{X} \psi(\tilde{u} - \phi + C) \theta^{n}_{\tilde{v}} \\
    &\leq \int_{X} \psi(\tilde{u} - \phi)\theta^{n}_{\tilde{v}} + \int_{X} \psi(C) \theta^{n}_{\tilde{v}} < \infty
\end{align*}
\end{proof}

\begin{cor} \label{cor: I psi is finite}
The proposed quasi-metric  $I_{\psi}$ (see Equation~\eqref{eq: proposed quasi-metric}) on $\E_{\psi}(X,\theta,\phi)$ is always finite.
\end{cor}
\begin{proof}
Let $u,v \in \E_{\psi}(X,\theta,\phi)$. Then 
\begin{align*}
    I_{\psi}(u,v) &= \int_{X}\psi(u-v)(\theta^{n}_{u} + \theta^{n}_{v}) \\
    & \leq \int_{X} (\psi(u-\phi) + \psi(v - \phi) ) (\theta^{n}_{u} + \theta^{n}_{v}).
\end{align*}
Here we used \cite[Lemma 2.6]{darvas2021mabuchi} which states that for any $a,b \in \R$, and any $\psi \in \mathcal{W}^{-}$, we have $\psi(a+b) \leq \psi(a) + \psi(b)$. By Lemma~\ref{lem: integrability} all the terms in the expression above are finite.
\end{proof}

\begin{lem}[Domination Principle]\label{lem: domination principle} \emph{\cite[Proposition 3.11]{Darvasmonotonicity}}. Let $\phi \in \PSH(X,\theta)$ such that $\int_{X} \theta^{n}_{\phi} > 0$. Then for $u,v \in \E(X,\theta,\phi)$, if $\theta^{n}_{u} ( \{ u< v\}) = 0$, then $u \geq v$. 
\end{lem}

We also need the following slight generalization of \cite[Theorem 2.2]{darvas2019metric}, removing the assumption of uniform boundedness of $\chi_{k}$. 
\begin{lem} \label{thm: removing uniform boundedness assumption}
Let $\theta^{j}$, $j \in \{1,\dots, n\}$ be smooth closed $(1,1)$-forms on $X$ whose cohomology classes are big. Let $u_{j}, u_{j}^{k} \in \PSH(X,\theta^{j})$ are such that $u_{j}^{k}\to u_{j}$ in capacity as $k \to \infty$. Let $\chi_{k},\chi \geq 0$ be quasi-continuous functions such that $\chi_{k} \to \chi$ in capacity as $k \to \infty$. Then 
\[
\int_{X} \chi \theta^{1}_{u_{1}} \wedge \theta^{2}_{u_{2}} \wedge \dots \wedge \theta^{n}_{u_{n}} \leq \liminf_{k\to \infty} \int_{X} \chi_{k} \theta^{1}_{u_{1}^{k}} \wedge \theta^{2}_{u_{2}^{k}} \wedge \dots \wedge \theta^{n}_{u_{n}^{k}}.
\]
\end{lem}

\begin{proof}
Consider $\chi_{k,C} = \min(\chi_{k}, C)$. Then $\chi_{k,C}$ are uniformly bounded and quasi-continuous and $\chi_{k,C} \to \chi_{C}$ in capacity. Therefore, using \cite[Theorem 2.2]{darvas2019metric}
\[
\int_{X} \chi_{C} \theta^{1}_{u_{1}} \wedge \dots \wedge \theta^{n}_{u_{n}} \leq \liminf_{k\to \infty} \int_{X} \chi_{k,C} \theta^{1}_{u_{1}^{k}} \wedge \dots \wedge \theta^{n}_{u_{n}^{j}}  
\]
Since $\chi_{k,C} \leq \chi_{k}$, we have 
\[
 \int_{X} \chi_{C} \theta^{1}_{u_{1}} \wedge \dots \wedge \theta^{n}_{u_{n}} \leq \liminf_{k \to \infty} \int_{X} \chi_{k,C} \theta^{1}_{u_{1}^{k}} \wedge \dots \wedge \theta^{n}_{u_{n}^{k}} \leq \liminf_{k \to \infty} \int_{X} \chi_{k} \theta^{1}_{u_{1}^{k}} \wedge \dots \wedge \theta^{n}_{u_{n}^{k}} 
\]
As $\chi_{C} \nearrow \chi$ as $C \to \infty$ and applying monotone convergence theorem, we get 
\[
\int_{X} \chi \theta^{1}_{u_{1}} \wedge \dots \wedge \theta^{n}_{u_{n}} \leq \liminf_{k \to \infty} \int_{X} \chi_{k} \theta^{1}_{u_{1}^{k}} \wedge \dots \wedge \theta^{n}_{u_{n}^{k}} .
\]
\end{proof}

The following result is mentioned in \cite[Corollary 3.4]{thai2021complex}. We present a proof here for completeness. 
\begin{lem}\label{lem: stability} 
Let $\phi \in \PSH(X,\theta)$ such that $\int_{X}\theta^{n}_{\phi} > 0$. Let $u_{j} \in \E_{\psi}(X,\theta,\phi)$ such that 
\[
\sup_{j}\int_{X}\psi(u_{j}  - \phi)\theta^{n}_{u_{j}} < \infty.
\]
If $u_{j} \to u$ in $L^{1}(\om^{n})$ for some $u \in \PSH(X,\theta)$, then $u \in \E_{\psi}(X,\theta,\phi)$. 
\end{lem}

\begin{proof}
Since there exists $A$ such that $\sup_{X}u_{j} \leq A$ for all $j$ (see \cite[Proposition 8.4]{guedj2017degenerate}), and the fact that $\psi(a+b) \leq \psi(a) + \psi(b)$ for any $a,b\in\R$ we can assume without loss of generality that $u_{j}, u \leq \phi$ for all $j$.

First, we assume that $u_{j} \searrow u$. This implies that $u_{j} \to u$ in capacity. Also, $\psi(u_{k} - \phi) \geq 0$ are quasicontinuous functions which converge in capacity to $\psi(u-\phi)$. Thus by Theorem~\ref{thm: removing uniform boundedness assumption}, we get that 
\[
\int_{X} \psi(u - \phi) \theta^{n}_{u} \leq \liminf_{k\to \infty} \int_{X} \psi(u_{k} - \phi) \theta^{n}_{u_{k}} < \infty.
\]
In general, if $u_{j} \to u$ in $L^{1}(\om^{n})$ and if $v_{j} := (\sup_{k\geq j}u_{k})^{*}$, then $u_{j} \leq v_{j} \leq \phi$. Thus, by Lemma~\ref{lem: the fundamental inequality} we get 
\[
\sup_{j}\int_{X} \psi(v_{j} - \phi)\theta^{n}_{v_{j}} \leq 2^{n+1}\sup \int_{X}\psi(u_{j} - \phi)\theta^{n}_{u_{j}} < \infty. 
\]
Since $v_{j} \searrow u$, by the argument above, we get that
\[
\int_{X} \psi(u-\phi)\theta^{n}_{u} < \liminf_{j \to \infty} \int_{X} \psi(v_{j} -\phi)\theta^{n}_{v_{j}} < \infty.
\]
Although this shows that $u$ has finite energy, we are yet to show that $u$ has full mass. 
By \cite[Lemma 3.4]{Darvasmonotonicity} we just need to show that $\int_{\{u\leq \phi - C\}} \theta^{n}_{u^{C}} \to 0$ as $C \to \infty$. Here $u^{C} = \max(u, \phi - C)$. First, notice that $v_{j}^C \searrow u^{C}$ in  $j \to \infty$. 
\[
\int_{X}\psi(u^{C} - \phi)\theta^{n}_{u^{C}} \leq \liminf_{j} \int_{X} \psi(v_{j}^{C} - \phi)\theta^{n}_{v^{C}_{j}} \leq 2^{n+1} \liminf_{j} \int_{X} \psi(u_{j} - \phi)\theta^{n}_{u_{j}}.
\]
Thus 
\[
\int_{\{u \leq \phi - C\}} \theta^{n}_{u^{C}} \leq \frac{1}{\psi(C)} \int_{X} \psi(u^{C} - \phi)\theta^{n}_{u^{C}} \leq \frac{2^{n+1}}{\psi(C)} \sup_{j} \int_{X} \psi(u_{j} - \phi)\theta^{n}_{u_{j}} \to 0
\]
as $C \to \infty$. This finishes the proof that $u \in \E_{\psi}(X,\theta,\phi)$. 
\end{proof}

\subsection{The operator $P_{\theta}(u,v)$} \label{subsec: the operator Pthetauv}
The notion of model potential and model type singularity was introduced by Darvas-Di Nezza-Lu \cite{Darvasmonotonicity} to solve the Monge-Amp\`ere equation in the prescribed singularity setting. In particular, in \cite{darvaslogconcavity}, the authors showed that if $\phi$ is a model potential  and if $\mu$ is a non-pluripolar positive measure such that $\mu(X) = \int_{X} \theta^{n}_{\phi} > 0$, then there unique (upto a constant) $u \in \E(X,\theta,\phi)$ such that $\theta^{n}_{u} = \mu$. 

Recall that for an upper semicontinuous function $f$, we define $P_{\theta}(f)$ to be the largest $\theta$-psh function lying below $f$. In particular, 
\[
P_{\theta}(f):= \sup\{ v \in \PSH(X,\theta) : v \leq f\}.
\]
When $u,v \in \PSH(X,\theta)$, we say $P_{\theta}(u,v) := P_{\theta}(\min(u,v))$. For $\phi \in \PSH(X,\theta)$, the envelope of its singularity type is defined as 
\[
P_{\theta}[\phi] := \sup\{ v \in \PSH(X,\theta,\phi) : v \leq 0\}.
\]

A potential $\phi \in \PSH(X,\theta)$ is called a \emph{model potential} if $\phi = P_{\theta}[\phi]$. Their importance in understanding the space $\E(X,\theta,\phi)$ and $\PSH(X,\theta)$ and solving the Monge-Amp\`ere equations is described in \cite{Darvasmonotonicity} and \cite{darvaslogconcavity}. In the rest of the section $\phi$ is a model potential. 

The following part is adapted from \cite{Darvas2017OnTS}. We will use the following lemma repeatedly in the arguments ahead. 

\begin{lem}\label{lem: lemma from darvas monotonicity}
Let $\mu$ be a non-pluripolar measure such that $0 < \mu(X) < \infty$ and $\lambda > 0$. Let $u ,v \in \E(X,\theta, \phi)$ satisfy 
\[
\theta^{n}_{u} \geq e^{\lambda u-w}\mu , \theta^{n}_{v} \leq e^{\lambda v - w}\mu
\]
for some Borel measurable function $w : X \to \R \cup \{ -\infty\}$ that is bounded from above. Then $u \leq v$ on $X$. 
\end{lem}

\begin{proof}
The proof is same as in \cite[Lemma 2.5]{Darvas2017OnTS}, adapted to the prescribed singularity case. Using the comparison principle, we have
\[
\int_{\{v<u\}} e^{\lambda u- w} \mu \leq \int_{\{v < u\}} \theta^{n}_{u} \leq \int_{\{v < u\}} \theta^{n}_{v} \leq \int_{\{v < u\}} e^{\lambda v - w}\mu  \leq \int_{\{v < u\}} e^{\lambda u - w}d\mu.
\]
Thus all the expressions are equal and we get 
\[
\int_{\{v < u \}}\left( e^{\lambda u} - e^{\lambda v}\right)e^{-w}\mu = 0.
\]
Since $w$ is bounded from above, $e^{-w}$ is never 0. Thus $\mu\{ v< u\} = 0$ and hence $\theta^{n}_{v}(\{v < u\}) = 0$. Now the domination principle (see Lemma~\ref{lem: domination principle}) implies that $u \leq v$. 
\end{proof}

\begin{thm}\label{thm: operator P is stable on E psi}
Let $u,v \in \E_{\psi}(X,\theta, \phi)$. Then 
\[
P_{\theta}(u,v) :=  \sup\{ w \in \PSH(X,\theta) : w \leq \min(u,v)\} \in \E_{\psi}(X,\theta,\phi).
\]
In particular, if $u,v \in \E(X,\theta, \phi)$, then $P_{\theta}(u,v) \in \E(X,\theta, \phi)$. 
\end{thm}

\begin{proof}
Let $u_{j} = \max(u, \phi - j)$ and $v_{j} = \max(v, \phi - j)$. Then $u_{j}$ and $v_{j}$ have the same singularity type as $\phi$. Thus by Lemma~\ref{lem: intermediary for operator P} below, there is a unique function $\varphi_{j} \in \E(X,\theta,\phi)$ with the same singularity type as $\phi$ such that 
\begin{equation}\label{eq: def of varphi j}
    \theta^{n}_{\varphi_{j}} = e^{\varphi_{j} - u_{j}}\theta^{n}_{u_{j}} + e^{\varphi_{j}- v_{j}}\theta^{n}_{v_{j}}.
\end{equation}

Notice that $\theta^{n}_{\varphi_{j}} \geq e^{\varphi_{j} - u_{j}} \theta^{n}_{u_{j}}$. Defining $\mu = e^{-u_{j}} \theta^{n}_{u_{j}}$, we see that Lemma~\ref{lem: lemma from darvas monotonicity} implies that $\varphi_{j} \leq u_{j}$, and similarly, $\varphi_{j} \leq v_{j}$. Therefore, $\varphi_{j} \leq \min(u_{j},v_{j})$. Consequently, $\varphi_{j} \leq P_{\theta}(u_{j}, v_{j})$.  Now we claim that 
\begin{equation}\label{eq: desired}
    \sup_{j} \int_{X} \psi(\varphi_{j} - \phi)\theta^{n}_{\varphi_{j}} < \infty. 
\end{equation}

By Equation~\eqref{eq: def of varphi j}, it is enough to show that
\begin{equation}\label{eq: intermedieary for desired}
    \sup_{j} \int_{X}\psi(\varphi_{j} - \phi) e^{\varphi_{j} - u_{j}} \theta^{n}_{u_{j}} < \infty. 
\end{equation}

Again, using the fact that $\psi(a+b) \leq \psi(a)+\psi(b)$, we get
\begin{align*}
\int_{X} \psi(\varphi_{j} - \phi)e^{\varphi_{j} - u_{j}}\theta^{n}_{u_{j}} &\leq \int_{X} \psi(\varphi_{j} - u_{j})e^{\varphi_{j} - u_{j}} \theta^{n}_{u_{j}} + \int_{X} \psi(u_{j} - \phi)e^{\varphi_{j} - u_{j}}\theta^{n}_{u_{j}}.
\intertext{Since $\psi(t)e^{t} \leq C$ for some fixed $C$ and for all $t\leq 0$, observing $\varphi_{j} \leq u_{j}$ we get that}
&\leq C\int_{X} \theta^{n}_{\phi} + \int_{X} \psi(u_{j} - \phi) \theta^{n}_{u_{j}}.
\end{align*}
As $u \in \E_{\psi}(X,\theta,\phi)$, we get Equation~\eqref{eq: intermedieary for desired} and consequently Equation~\eqref{eq: desired}. 

Since $\varphi_{j} \leq u_{j} \leq u_{1}$, we get that $\sup_{X}\varphi_{j}$ is uniformly bounded. By the proof of Lemma~\ref{lem: decreasing sequence has finite energy in prescribed singularity case}, we also get that $\varphi_{j} \not \to -\infty$ uniformly. Thus upto choosing a subsequence, we get that there exists $\varphi \in \PSH(X,\theta)$ such that $\varphi_{j} \to \varphi$ in $L^{1}(\om^{n})$. By Lemma~\ref{lem: stability} we get that $\varphi \in \E_{\psi}(X,\theta,\phi)$. Since $\varphi_{j} \leq P_{\theta}(u_{j},v_{j})$, we get that $\varphi \leq P_{\theta}(u,v)$. Thus Lemma~\ref{lem: the fundamental inequality} tells us that $P_{\theta}(u,v) \in \E_{\psi}(X,\theta,\phi)$.

\end{proof}

\begin{lem}\label{lem: intermediary for operator P} Let $u,v \in \PSH(X,\theta,\phi)$ such that $u,v$ have the same singularity type as $\phi$. Then there exists a unique $\varphi \in \PSH(X,\theta, \phi)$ with the same singularity type as $\phi$ such that 
\[
\theta^{n}_{\varphi} = e^{\varphi - u}\theta^{n}_{u} + e^{\varphi - v}\theta^{n}_{v}.
\]
\end{lem}

\begin{proof}
First we show uniqueness. Let $\tilde{\varphi}$ be another solution. Let $\mu = e^{u}\theta^{n}_{v} + e^{v}\theta^{n}_{u}$ and $w  = u+ v$. Then $\mu(X) < \infty$ and $w$ is bounded from above. Notice that  $\theta^{n}_{\varphi} = e^{\varphi - w}\mu$ and $\theta^{n}_{\tilde{\varphi}} = e^{\tilde{\varphi} - w}\mu$. By Lemma~\ref{lem: lemma from darvas monotonicity} we get  that $\varphi = \tilde{\varphi}$.

Let $u_{j} = \max(u, -j)$. Note that $u_{j}$ is no longer a $\theta$-psh function, but now it is a bounded function. Consider the measure 
\[
\mu_{j} = e^{-u_{j}} \theta^{n}_{u} + e^{-v_{j}} \theta^{n}_{v}.
\]
Then $\mu_{j}$ is a non-pluripolar positive measure. By \cite[Theorem 4.23]{Darvasmonotonicity}, we get that there exists a unique $\varphi_{j} \in \E^{1}(X,\theta,\phi)$ such that 
\[
\theta^{n}_{\varphi_{j}} =  e^{\varphi_{j}}\mu_{j}.
\]
Since $u$ and $v$ have the same singularity type, there exists $C$ such that $\sup_{X} |u-v| \leq 2C$. Consider the function 
\[
w = \frac{u+v}{2} - C - (n+1)\ln 2.
\]
Observing $\theta^{n}_{u}, \theta^{n}_{v} \leq 2^{n}\theta^{n}_{w}$, we get 
\[
\theta^{n}_{w} \geq e^{w}\mu_{j}.
\]
Now Lemma~\ref{lem: lemma from darvas monotonicity} implies that $w \leq \varphi_{j}$. As $w$ has relative minimal singularity type we obtain $\varphi_{j}$ has relative minimal singularity type as well. Notice that for $j \geq k$, we have $\mu_{j} \geq \mu_{k}$. Thus,
\[
\theta^{n}_{\varphi_{j}} \geq e^{\varphi_{j}} \mu_{k} \qquad \text{and} \qquad \theta^{n}_{\varphi_{k}} = e^{\varphi_{k}} \mu_{k}.
\]
Lemma~\ref{lem: lemma from darvas monotonicity} again shows that $\varphi_{k} \geq \varphi_{j}$. Thus $\varphi_{j}$ is decreases as $j \to \infty$. Let $\varphi = \lim_{j} \varphi_{j}$. Then $w\leq \varphi$ thus $\varphi$ has the relative minimal singularity type and by continuity of non-pluripolar Monge-Amp\`ere operator under decreasing sequences, we get that 
\[
\theta^{n}_{\varphi} = e^{\varphi - u}\theta^{n}_{u} + e^{\varphi - v}\theta^{n}_{v}.
\]
\end{proof}

\subsection{Metrics from a quasi-metric} \label{subsec: metrics from a quasi-metric}

If we relax the condition of the triangle inequality from the definition of a metric space, we obtain what we call a quasi-metric space. 
\begin{mydef}[Quasi-metric space]
Given a set $X$ a function $\rho : X \times X \to [0,\infty)$ is a quasi-metric if it satisfies 
\begin{enumerate}
    \item (Non-degeneracy) For any $x,y \in X$, $\rho(x,y) = 0$ iff $x = y$.
    \item (Symmetry) For all $x,y \in X$, $\rho(x,y) = \rho(y,x)$.  
    \item (Quasi-triangle inequality)There exists $C \geq 1$ such that  $\rho(x,y) \leq C (\rho(x,z) + \rho(y,z))$ for all $x,y,z \in X$. 
\end{enumerate}

In \cite{Paluszynski2009OnQA} the authors show that given any quasi-metric, we can construct a metric comparable to the quasi-metric using a $p$-chain approach. In particular, if we consider $d_{p} : X \times X \to [0 , \infty)$ given by
 \[
 d_{p}(x,y) = \inf \left\{ \sum_{i=0}\rho(x_{n},x_{n+1})^{p} :  x_{0}, \dots, x_{n+1} \in X \text{ such that }x = x_{0}, y = x_{n+1}\right\}
 \]
then $d_{p}$ is symmetric, satisfies triangle inequality, but in general $d_{p}(x,y) = 0$ even if $x\neq y$. But if $0<p\leq 1$ is chosen such that $(2C)^{p} = 2$, then $d_{p}$ is non-degenerate and moreover, 
\[
d_{p}(x,y) \leq \rho(x,y)^{p} \leq 4 d_{p}(x,y).
\]
This shows that if for $x_{n}, x \in X$ we have $\rho(x_{n}, x) \to 0$ iff $d_{p}(x_{k},x) \to 0$. Thus $\rho$ and $d_{p}$ induce the same topology. 
\end{mydef}

\section{Quasi-metric on $\E_{\psi}(X,\theta,\phi)$}\label{sec: quasi-metric space}
In this section we  prove that the expression 
\[
I_{\psi}(u,v) = \int_{X} \psi(u-v)(\theta^{n}_{u} + \theta^{n}_{v})
\]
is a quasi-metric on $\E_{\psi}(X,\theta,\phi)$ where $\phi \in \PSH(X,\theta)$ is such that $\int_{X} \theta^{n}_{\phi} > 0$. 

\begin{thm}[Quasi-triangle inequality]\label{thm : quasi-triangle inequality in the prescribed singularity case}
Let $\phi \in \PSH(X,\theta)$ be such that $\int_{X} \theta^{n}_{\phi} > 0$, then for any $u,v,w \in \E_{\psi}(X,\theta,\phi)$, the functional
\[
I_{\psi}(u,v)  = \int_{X} \psi(u-v)(\theta^{n}_{u} + \theta^{n}_{v})
\]
satisfies 
\[
I_{\psi}(u,v) \leq C (I_{\psi}(u,w) + I_{\psi}(v,w)) 
\]
for $C = 8\cdot 3^{n+1}$. 
\end{thm}

\begin{proof}
The proof is inspired from the proof in \cite[Theorem 1.6]{Guedj2019PlurisubharmonicEA}. We start with 
\begin{align}
    \int_{X}\psi(u-v)\theta^{n}_{u} &= \int_{0}^{\infty} \theta^{n}_{u}(\{\psi(u-v) > t\})dt.
    \intertext{Changing the variable $t = \psi(2s)$, we get $dt = 2 \psi'(2s)ds$. Thus}
    &= 2 \int_{0}^{\infty}\psi'(2s)\theta^{n}_{u}(\{\psi(u-v) > \psi(2s)\})ds\\
    &= 2 \int_{0}^{\infty}\psi'(2s) \theta^{n}_{u}(\{|u-v| > 2s\})ds. \label{eq: integral}
    \intertext{Observe that $\{w - s \leq  u < v-2s\} \subset \{ w < \frac{u+2v}{3} - \frac{s}{3}\}$ and $\{u<v-2s\} \subset \{u < w - s \} \cup \{w < \frac{u+2v}{3}- \frac{s}{3}\}$ to obtain }
    \theta^{n}_{u}(\{u< v -2s\}) &\leq \theta^{n}_{u}(\{ u < w - s\}) + \theta^{n}_{u}\left(\left\{ w < \frac{u+2v}{3} - \frac{s}{3}\right\}\right). 
    \intertext{Since $\theta^{n}_{u} \leq 3^{n}\theta^{n}_{\frac{u+2v}{3}}$ and $\E(X,\theta,\phi)$ is convex \cite[Corollary 3.15]{Darvasmonotonicity}, we get }
    &\leq \theta^{n}_{u}(\{u < w -s \}) +  3^{n}\theta^{n}_{\frac{u+2v}{3}}\left( \left \{ w < \frac{u+2v}{3} - \frac{s}{3}\right\} \right).
    \intertext{Now Lemma~\ref{lem: comparison principle} gives} 
    &\leq \theta^{n}_{u} (\{u < w-s\}) + 3^{n}\theta^{n}_{w}\left( \left\{ w < \frac{u+2v}{3} - \frac{s}{3}\right\} \right).\label{eq: for u<v}
    \intertext{Again by the comparison principle we have}
    \theta^{n}_{u}(\{ v< u - 2s\}) &\leq \theta^{n}_{v}(\{ v < u - 2s\}). \notag
    \intertext{By a similar computation as in Equation~\eqref{eq: for u<v} we get }
    &\leq \theta^{n}_{v}(\{v < w - s \}) + 3^{n}\theta^{n}_{w}\left( \left\{ w < \frac{ v + 2u}{3} - \frac{s}{3} \right\} \right). \label{eq: for v<u}
    \intertext{Combining Equation~\eqref{eq: for u<v} and~\eqref{eq: for v<u} we get}
    \theta^{n}_{u}(\{|u-v|> 2s\}) &\leq \theta^{n}_{u}(\{u<v-2s\}) + \theta^{n}_{u}(\{ v< u-2s\})\notag \\
    &\leq \theta^{n}_{u} (\{u < w-s\}) + 3^{n}\theta^{n}_{w}\left( \left\{ w < \frac{u+2v}{3} - \frac{s}{3}\right\} \right)\notag \\ 
    &\quad +\theta^{n}_{v}(\{v < w - s \}) + 3^{n}\theta^{n}_{w}\left( \left\{ w < \frac{ v + 2u}{3} - \frac{s}{3} \right\} \right)\notag\\
    &\leq \theta^{n}_{u}(\{|u-w| > s\}) + 3^{n}\theta^{n}_{w}\left( \left\{ \left| w - \frac{u+2v}{3}\right| > \frac{s}{3} \right\} \right) \notag\\
    &+\theta^{n}_{v}(\{|v-w|>s\}) + 3^{n}\theta^{n}_{w}\left( \left\{ \left| w - \frac{v+2u}{3}\right| > \frac{s}{3} \right\} \right). \label{eq: for |u-v|}
    \intertext{Combining Equations~\eqref{eq: integral} and~\eqref{eq: for |u-v|}, we get }
    \int_{X}\psi(u-v)\theta^{n}_{u}&\leq 2\int_{0}^{\infty} \psi'(2s)\theta^{n}_{u}(\{|u-w|>s\})ds \notag \\ 
    &+ 2 \int_{0}^{\infty}\psi'(2s) \theta^{n}_{v}(\{|v-w|>s\})ds \notag \\
    &+2\cdot3^{n}\int_{0}^{\infty} \psi'(2s) \theta^{n}_{w}\left( \left\{ \left| w - \frac{u+2v}{3}\right| > \frac{s}{3} \right\} \right)ds\notag \\
    &+2\cdot3^{n} \int_{0}^{\infty} \psi'(2s) \theta^{n}_{w}\left( \left\{ \left| w - \frac{v+2u}{3}\right| > \frac{s}{3} \right\} \right)ds. \label{eq: expression for half integral}
    \intertext{Since $\psi$ is concave, therefore the slope is decreasing. Hence $\psi'(2s) \leq \psi'(s)$ and $\psi'(2s) < \psi'(s/3)$. Using this in Equation~\eqref{eq: expression for half integral} we get }
    &\leq 2\int_{0}^{\infty} \psi'(s)\theta^{n}_{u}(\{|u-w|>s\})ds \notag \\ 
    &+ 2 \int_{0}^{\infty}\psi'(s) \theta^{n}_{v}(\{|v-w|>s\})ds \notag \\
    &+2\cdot3^{n}\int_{0}^{\infty} \psi'(s/3) \theta^{n}_{w}\left( \left\{ \left| w - \frac{u+2v}{3}\right| > \frac{s}{3} \right\} \right)ds\notag \\
    &+2\cdot3^{n} \int_{0}^{\infty} \psi'(s/3) \theta^{n}_{w}\left( \left\{ \left| w - \frac{v+2u}{3}\right| > \frac{s}{3} \right\} \right)ds.
    \intertext{Changing the variable again $t = \psi(s)$ in the first two terms and $t = \psi(s/3)$ in the last two terms, we get}
    &= 2\int_{0}^{\infty}\theta^{n}_{u}(\{ \psi(u-w) > t\}) dt \notag\\ &+ 2 \int_{0}^{\infty} \theta^{n}_{v}(\{ \psi(v-w) > t\})dt \notag\\
    & +2\cdot 3^{n+1} \int_{0}^{\infty} \theta^{n}_{w} \left( \left\{ \psi\left( w - \frac{u+2v}{3} \right) > t \right\} \right) dt\notag \\
    &+ 2\cdot 3^{n+1} \int_{0}^{\infty} \theta^{n}_{w}\left( \left\{ \psi\left( w - \frac{v+ 2u}{3} \right) > t\right\} \right) dt. \\
    \int_{X}\psi(u-v)\theta^{n}_{u} &\leq 2\int_{X}\psi(u-w)\theta^{n}_{u} + 2\int_{X} \psi(v-w) \theta^{n}_{v} \notag \\
    &+ 2\cdot 3^{n+1} \int_{X} \psi\left( w - \frac{u+2v}{3}\right) \theta^{n}_{w} \notag \\
    &+2\cdot 3^{n+1} \int_{X} \psi\left( w - \frac{v+2u}{3} \right) \theta^{n}_{w}.
    \intertext{Using the fact that $\psi(a+b) \leq \psi(a) + \psi(b)$ for any $a,b \in \R$, (see \cite[Lemma 2.6]{darvas2021mabuchi})  we get} 
    &\leq 2\int_{X}\psi(u-w)\theta^{n}_{u} + 2\int_{X} \psi(v-w) \theta^{n}_{v} \notag \\
    &+ 2\cdot3^{n+1} \int_{X} \left( \psi\left( \frac{w-u}{3}\right) + \psi\left(\frac{2(w-v)}{3}\right)\right) \theta^{n}_{w} \notag \\ 
    &+ 2\cdot 3^{n+1}\int_{X} \left( \psi\left( \frac{w-v}{3} \right) + \psi\left( \frac{2(w-u)}{3}\right) \right)\theta^{n}_{w}.
    \intertext{Since $\psi$ is increasing in $(0,\infty)$ and symmetric, we get that $\psi(x/3) \leq \psi(x)$ and $\psi(2x/3) \leq \psi(x)$ for any $x \in \R$, thus}
    &\leq 2\int_{X}\psi(u-w)\theta^{n}_{u} +2\int_{X}\psi(v-w)\theta^{n}_{v} \notag \\
    &+ 4\cdot3^{n+1}\int_{X} \psi(w-u)\theta^{n}_{w} + 4\cdot 3^{n+1}\int_{X} \psi(w-v)\theta^{n}_{w}\\
    &\leq 4\cdot 3^{n+1} (I(u,w) + I(v,w)).
    \intertext{A similar computation for $\int_{X}\psi(u-v)\theta^{n}_{v}$ gives}
    I(u,v) &\leq 8\cdot 3^{n+1}(I(u,w) + I(v,w)).
 \end{align}
 This finishes the proof of quasi-triangle inequality.
\end{proof}

\begin{thm}[Non-degeneracy]\label{thm: non-degeneracy in the prescribed singularity case}
If $u,v \in \E_{\psi}(X,\theta, \phi)$, such that $I_{\psi}(u,v) = 0$, then $u = v$. 
\end{thm}
\begin{proof}
Let $I_{\psi}(u,v) = 0$. Then 
\[
\int_{X} \psi(u-v) \theta^{n}_{u} + \int_{X} \psi(u-v) \theta^{n}_{v} = 0.
\]
Thus we have $u = v$ almost everywhere with respect to $\theta^{n}_{u}$ and $\theta^{n}_{v}$. Using Lemma~\ref{lem: domination principle}, we get that $u = v$. 
\end{proof}

The symmetry of $I_{\psi}(u,v)$ follows from the fact that the weight function is an even function. This finishes the proof of the fact that $I_{\psi}$ is a quasi-metric on $\E_{\psi}(X,\theta,\phi)$. 

\section{Completeness} \label{sec: Completeness}

In this section we  finish the proof of Theorem~\ref{thm: complete metric space structure on prescribed singularity case} by showing that the quasi-metric is complete. 

We need the following lemma to construct a monotone sequence and show that it converges. 
\begin{thm}[Pythagorean Identity]\label{thm: pythagorean identity for max in prescribed singularity}
Let $\phi \in \PSH(X,\theta)$ and let $u,v \in \E_{\psi}(X,\theta, \phi)$. We know that $\max(u,v) \in \E_{\psi}(X,\theta, \phi)$. Moreover, they satisfy
\[
I_{\psi}(u,v) = I_{\psi}(\max(u,v),u) + I_{\psi}(\max(u,v),v).
\]
\end{thm}

\begin{proof}
\begin{align*}
    I_{\psi}(\max(u,v),u) &= \int_{X} \psi( \max(u,v) - u)(\theta^{n}_{\max(u,v)} + \theta^{n}_{u}) \\
    &= \int_{\{v> u\}} \psi(v-u) (\theta_{\max(u,v)}^{n} + \theta^{n}_{u}). 
    \intertext{Since $u \mapsto \theta^{n}_{u}$ is plurifine local, we have $\mathds{1}_{\{v >u\}} \theta^{n}_{\max(u,v)} = \mathds{1}_{\{ v > u\}}\theta^{n}_{v}$}
    &= \int_{\{v > u\}}\psi(v-u)(\theta^{n}_{v} + \theta^{n}_{u}).
    \intertext{Similarly, }
    I_{\psi}(\max(u,v), v) &= \int_{X}\psi(\max(u,v) - v)(\theta^{n}_{\max(u,v)} + \theta^{n}_{v})
    \intertext{and the same computation as before gives}
    &= \int_{\{u > v\}} \psi(u-v) (\theta^{n}_{u} + \theta^{n}_{v}).
    \intertext{Adding the two gives}
    I_{\psi}(\max(u,v) ,u) + I_{\psi}(\max(u,v),v) &= \int_{X}\psi(u-v)(\theta^{n}_{u} + \theta^{n}_{v}) \\
    &= I_{\psi}(u,v).
\end{align*}

\end{proof}

\begin{prop}\label{convergence with monotone sequences in prescribed singularity setting}
If $u_{j}, u \in \E_{\psi}(X,\theta, \phi)$ such that $u_{j} \searrow u$ or $u_{j} \nearrow u$, then 
\[
\int_{X}\psi(u_{j} - u) \theta^{n}_{u_{j}} \to 0.
\]
Consequently, by the the monotone convergence theorem $I_{\psi}(u_{j},u) \to 0$. 
\end{prop}
\begin{proof}
Without loss of generality, we can assume that $u_{j}, u \leq \phi \leq 0$. First, we assume that $\phi - L \leq u_{j}, u \leq \phi$. We also assume that $u_{j} \searrow u$ and the proof for $u_{j} \nearrow u$ is similar. Thus $0 \leq u_{j} - u \leq L$, which implies $0 \leq \psi(u_{j} - u) \leq \psi(L)$.  For some large enough $A$, $\theta \leq A\om$. Thus $\PSH(X,\theta) \subset \PSH(X,A\om)$, and we get that $u_{j}, u \in \PSH(X,A\om)$. Hence they are quasi-continuous (at least with respect to $\CAP_{\om}$). So we can find an open set $O$ such that $u_{j}, u$ are continuous on $X \setminus O$ and $\CAP_{\om}(O) < \ee$. The next part follows the proof in \cite[Corollary 2.12]{Darvas2017OnTS} adapted to our case. 
We have 
\[
\int_{X} \psi(u_{j} - u) \theta^{n}_{u_{j}} = \int_{X\setminus O} \psi(u_{j} - u) \theta^{n}_{u_{j}} + \int_{O} \psi(u_{j} -u) \theta^{n}_{u_{j}}.
\]
Since $X\setminus O$ is compact and $\psi(u_{j} -u)$ are continuous functions decreasing to $0$, by Dini's theorem, we get that the convergence is uniform. So for large enough $j$, $\psi(u_{j} - u) \leq \ee$ on $X\setminus O$. 

Moreover, 
\begin{align*}
    \int_{O} \psi(u_{j} - u) \theta^{n}_{u_{j}} &\leq \psi(L) \int_{O} \theta^{n}_{u_{j}} \\
    &\leq \psi(L) L^{n}\CAP_{\theta,\phi}(O) \\
    &\leq \psi(L)L^{n}f(\CAP_{\om}(O))\\
    &\leq \psi(L)L^{n} f(\ee).
\end{align*}
Here we used \cite[Theorem 1.1]{lucapacities} with $(\theta_{1}, \psi_{1}) = (\theta, \phi)$ and $(\theta_{2},\psi_{2}) = (\om, 0)$ and thus $\CAP_{\theta, \phi} \leq f(\CAP_{\om})$ where $f : \R^{+} \to \R^{+}$ is continuous and $f(0) = 0$.  Combining these two results we get that for large enough $j$,
\[
\int_{X}\psi(u_{j} - u)\theta^{n}_{u_{j}} \leq \ee \text{vol}(\theta^{n}) + \tilde{C} f(\ee). 
\]
Therefore, $\lim_{j\to \infty} \int_{X} \psi(u_{j} - u)\theta^{n}_{u_{j}} \to 0$. 

Now we remove the assumption that $\phi - L \leq u_{j} \leq \phi$. Let $u_{j}^{L} = \max\{ u_{j}, \phi - L\}$ and $u^{L} = \max\{ u , \phi - L\}$.
We want to show that 
\[
\left| \int_{X} \psi(u_{j} - u)\theta^{n}_{u_{j}} - \int_{X}\psi(u_{j}^{L} - u^{L})\theta^{n}_{u_{j}^{L}} \right| \to 0 \text{ as } L\to \infty
\]
uniformly with respect to $j$. Since $u_{j} \searrow u$, we have that $u_{j} \geq u$. On the set $\{ u > \phi-L\}$, we have $u^{L} = u$ and $u_{j}^{L} = u_{j}$. Using the fact that Monge-Amp\`ere measures are plurifine, we get that 
\[
\int_{\{u >\phi -L\}} \psi(u_{j} - u)\theta^{n}_{u_{j}} = \int_{\{ u > \phi-L\}} \psi(u_{j}^{L} - u^{L})\theta^{n}_{u_{j}^{L}}.
\]
Thus we need to show that 
\[
\left| \int_{ \{ u \leq \phi - L \}} \psi(u_{j} - u) \theta^{n}_{u_{j}} - \int_{\{ u \leq \phi - L\} } \psi(u_{j}^{L} - u^{L}) \theta^{n}_{u_{j}^{L}}\right| \to 0.
\]
We will show that both the terms go to 0 as $L \to \infty$ independent of $j$. 

Notice that since $u_{j} \geq u$ we have $u_{j} - u = (\phi - u) - (\phi - u_{j}) \leq \phi - u$. Since $\psi$ is increasing in $(0,\infty)$ we have $\psi(u_{j} - u) \leq \psi(\phi - u) = \psi(u - \phi)$. Now we can construct a weight $\tilde{\psi}$ such that $\tilde{\psi}$ is still even, $\tilde{\psi}(0) = 0$ and $\tilde{\psi}_{|(0,\infty)}$ is smooth increasing and concave such that $\psi(t) \leq \tilde{\psi}(t)$ and $\psi(t)/\tilde{\psi}(t) \searrow 0$ as $t \to \infty$. Moreover, we require that $u \in \E_{\tilde{\psi}}(X,\theta, \phi)$. Since $u_{j} \geq u$, we have $u_{j} \in \E_{\tilde{\psi}}(X,\theta, \phi)$ as well. Now
\begin{align*}
    \int_{\{ u \leq \phi - L\}}\psi(u_{j} - u)\theta^{n}_{u_{j}} &\leq \int_{\{ u \leq \phi - L\}} \psi(u - \phi) \theta^{n}_{u_{j}}  \\
    &\leq \frac{\psi(L)}{\tilde{\psi}(L)} \int_{\{u \leq \phi - L\}} \tilde{\psi}(u-\phi) \theta^{n}_{u_{j}}.
    \intertext{Since $u,u_{j} \in \E_{\tilde{\psi}}(X,\theta, \phi)$, Lemma~\ref{lem: integrability} tells us}
    &\leq 2\frac{\psi(L)}{\tilde{\psi}(L)}\left( \int_{X} \tilde{\psi}(u - \phi) \theta^{n}_{u} + \int_{X} \tilde{\psi}(u_{j} - \phi) \theta^{n}_{u_{j}}\right).
    \intertext{Using Lemma~\ref{lem: the fundamental inequality} and noticing that $u\leq u_{j} \leq \phi$ we get that }
    &\leq 2(2^{n+1}+1)\frac{\psi(L)}{\tilde{\psi}(L)} \int_{X} \tilde{\psi}(u-\phi)\theta^{n}_{u}.
\end{align*}
Since $u$ has finite $\tilde{\psi}$ energy, $\int_{X} \tilde{\psi}(u-\phi)\theta^{n}_{u}$ is finite and $\psi(L)/\tilde{\psi}(L) \to 0$ as $L \to \infty$ independent of $j$. The same proof shows that $\int_{\{ u\leq \phi - L\}} \psi(u_{j}^{L}  - u^{L}) \theta^{n}_{u_{j}^{L}} \to 0$ independent of $j$. 

We can also modify the proof to show that it works when $u_{j}\nearrow u$. 
\end{proof}

\begin{lem}\label{lem: decreasing cauchy sequence in prescribed singularity}
Let $u_{k} \in \E_{\psi}(X,\theta,\phi)$ be a decreasing  that  is $I_{\psi}$-Cauchy: for any $\ee > 0$, there exists $N$ such that $I_{\psi}(u_{j},u_{k}) \leq \ee$ for $j,k \geq N$. Then $\lim_{k\to\infty}u_{k}=:u \in \E_{\psi}(X,\theta,\phi)$ and $I_{\psi}(u_{k},u) \to 0$. 
\end{lem}

\begin{proof}
Find a subsequence of $\{u_{k}\}$ and still denote it by $\{u_{k}\}$ such that 
\[
I_{\psi}(u_{k}, u_{k+1}) \leq C^{-2k}
\]
where $C$ is the same constant as in Theorem~\ref{thm : quasi-triangle inequality in the prescribed singularity case}. Let $u = \lim_{k} u_{k}$. We have to show $u_{k} \not\equiv -\infty$. We have $I_{\psi}(u_{k}, \phi)$ is bounded as $k \to \infty$. To see this, repeated application of Theorem~\ref{thm : quasi-triangle inequality in the prescribed singularity case} gives
\begin{align*}
    I_{\psi}(\phi, u_{k}) &\leq C(I_{\psi}(\phi, u_{1}) + I_{\psi}(u_{1}, u_{k})) \\
    &\leq CI_{\psi}(\phi, u_{1}) + C^{2}I_{\psi}(u_{1},u_{2}) + C^{2}I_{\psi}(u_{2}, u_{k}).
    \intertext{Doing this $k$ times we get}
    &\leq CI_{\psi}(\phi, u_{1}) + \sum_{j=2}^{k}C^{j}I_{\psi}(u_{j-1},u_{j}) \\
    &\leq CI_{\psi}(\phi, u_{1}) + \sum_{j=2}^{k} C^{j}C^{-2j+2} \\
    &\leq CI_{\psi}(\phi, u_{1}) +C^{2} \sum_{j=2}^{\infty}C^{-j} \\
    &= CI_{\psi}(\phi, u_{1}) + \frac{C}{C-1} = \tilde{C}.
\end{align*}
This gives us 
\[
\int_{X}\psi(u_{k} - \phi)\theta^{n}_{\phi} \leq I(\phi, u_{k}) \leq \tilde{C}.
\]
Without loss of generality we can assue that $u_{1} \leq \phi$ so $u_{k} \leq u_{1} \leq \phi$. Since $u_{k}$ is decreasing, we have $u_{k} - \phi$ is decreasing, thus $\psi(u_{k} - \phi)$ is increasing. Applying monotone convergence theorem we get that 
\[
\int_{X}\psi(u-\phi)\theta^{n}_{\phi} \leq \tilde{C}.
\]
Thus $u \not\equiv -\infty$. Hence $u \in \PSH(X,\theta)$. Since $u_{j} \searrow u$ and 
\[
\int_{X}\psi(u_{k} - \phi)\theta^{n}_{u_{k}} \leq I_{\psi}(u_{k}, \phi) \leq \tilde{C}
\]
using Lemma~\ref{lem: stability}      we get that $u \in \E_{\psi}(X,\theta)$. Now we need to show that $I_{\psi}(u_{k},u) \to 0$. This means 
\[
\int_{X}\psi(u_{k} - u) \theta^{n}_{u} + \int_{X}\psi(u_{k}- u)\theta^{n}_{u_{k}} \to 0.
\]
For the first integral, we notice that $u_{k} \geq u$, thus $u_{k} - u \geq 0$ and $u_{k} - u \searrow 0$ and thus $\psi(u_{k} - u) \searrow 0$. Thus we can apply the monotone convergence theorem to get 
\[
\int_{X}\psi(u_{k} - u) \theta^{n}_{u} \to 0.
\]
To show that 
\[
\int_{X} \psi(u_{k} - u)\theta^{n}_{u_{k}} \to 0
\]
we use Proposition~\ref{convergence with monotone sequences in prescribed singularity setting}. Thus $I_{\psi}(u_{k},u) \to 0$. 
\end{proof}

\begin{lem}\label{uniformly bounded from above in prescribed singularity case}
Let $\{u_{k}\} \in \E_{\psi}(X,\theta, \phi)$ be such that 
\[
I_{\psi}(u_{k},u_{k+1}) \leq C^{-2k}.
\]
Then $u_{k}$ are uniformly bounded from above.
\end{lem}
\begin{proof}
Let 
\[
v_{k} = \max\{u_{1}, \dots, u_{k}\}.
\]
Then repeated application of Theorem~\ref{thm : quasi-triangle inequality in the prescribed singularity case} and~\ref{thm: pythagorean identity for max in prescribed singularity} gives
\begin{align*}
    I_{\psi}(\phi,v_{k}) &\leq C I_{\psi}(\phi,u_{1})  + C I_{\psi}(u_{1}, \max\{ u_{1}, \dots, u_{k}\}) \\
    &\leq C I_{\psi}(\phi,u_{1}) + C I_{\psi}(u_{1}, \max\{u_{2}, \dots, u_{k}\}) \\
    &\leq C I_{\psi}(\phi,u_{1}) + C^{2} I_{\psi}(u_{1},u_{2}) + C^{2} I_{\psi}(u_{2}, \max\{u_{2}, \dots, u_{k}\}) \\
    &\leq C I_{\psi}(\phi,u_{1}) + C^{2} I_{\psi}(u_{1},u_{2})  + C^{2} I_{\psi}(u_{2}, \max\{u_{3}, \dots, u_{k}\}).
    \intertext{Doing this $k$ times we get}
    &\leq C I_{\psi}(\phi,u_{1}) + C^{2} I_{\psi}(u_{1}, u_{2}) + \dots + C^{k}I_{\psi}(u_{k-1},u_{k}) \\
    &\leq \tilde{C}.
\end{align*}
Therefore, 
\[
\int_{X}\psi(v_{k} - \phi)\theta^{n}_{\phi} \leq \tilde{C}.
\]
Let $w_{k} = \max\{\phi, v_{k}\} \in \E_{\psi}(X,\theta)$. Then $w_{k} - \phi \geq 0$ and as $v_{k}$ is pointwise increasing, we get that $w_{k}-\phi$ is pointwise increasing. As $\psi$ is increasing on $[0,\infty)$, we get that $\psi(w_{k} - \phi)$ is pointwise increasing. Let $w := \lim_{k\to\infty}w_{k}$. Then $\psi(w-\phi) = \lim_{k\to \infty}\nearrow \psi(w_{k} - \phi)$. The monotone convergence theorem implies
\[
\int_{X} \psi(w - \phi) \theta^{n}_{\phi} = \lim_{k\to\infty} \int_{X} \psi(w_{k} - \phi)\theta^{n}_{\phi} \leq \int_{X} \psi(v_{k} - \phi)\theta^{n}_{\phi}\leq \tilde{C}.
\]
Thus the set $K = \{ \psi(w - \phi) \leq (\tilde{C} + 1)/\int_{X}\theta^{n}_{\phi}\}$ has positive $\theta^{n}_{\phi}$-measure. Thus $K$ is not $\PSH(X,\om)$-polar.  We can find a constant $A$ such that $\theta \leq A\om$. Then  $\PSH(X,\theta) \subset \PSH(X,A\om)$.  

Consider the set 
\[
\mathcal{F}_{K} = \{ \varphi \in \PSH(X, A\om) : \sup_{K}\varphi  = 0\}.
\]
Since $K$ is not $\PSH(X,A\om)$-polar, we know that $\mathcal{F}_{K}$ is compact in $L^{1}(\om^{n})$-topology. This follows from \cite[Theorem 4.7]{Guedj2004IntrinsicCO}. Let $\al_{k} = \sup_{K}v_{k}$. Let $e$ be such that $\psi(e) = (\tilde{C}  + 1)/\int_{X}\theta^{n}_{\phi}$. Then $K = \{w - \phi \leq e\}$. Notice that $\phi \leq 0$ gives
\[
\al_{k} = \sup_{K} v_{k} \leq \sup_{K}(v_{k} - \phi) \leq \sup_{K}(w - \phi) \leq e.
\]
Define $\tilde{v}_{k} = v_{k} - \al_{k}$. Then $\tilde{v}_{k} \in \PSH(X,\theta) \subset \PSH(X,A\om)$. Since $\sup_{K}\tilde{v}_{k} = 0$, we have $\tilde{v}_{k} \in \mathcal{F}_{K}$. Since $\mathcal{F}_{K}$ is relatively compact, we have $\{\tilde{v}_{k}\}$ is relatively compact in $L^{1}(X)$ topology and hence uniformly bounded from above. Therefore $v_{k} = \tilde{v}_{k} + \al_{k}$ are uniformly bounded from above and hence $u_{k} \leq v_{k}$ are uniformly bounded from above.
\end{proof}

The following theorem is the last step in proving completeness of the quasi-metric. 
\begin{thm}\label{thm: completeness in prescribed singularity setting}
Let $\{u_{j}\} \in \E_{\psi}(X,\theta, \phi)$ be an $I_{\psi}$-Cauchy sequence. Then there exists $u \in \E_{\psi}(X,\theta, \phi)$ such that $I_{\psi}(u_{j}, u) \to 0$.
\end{thm}

\begin{proof}
Extract a subsequence and still denote it by $\{u_{j}\}$ such that 
\[
I_{\psi}(u_{j}, u_{j+1}) \leq C^{-2j}
\]
where $C$ is the same constant as in Theorem~\ref{thm : quasi-triangle inequality in the prescribed singularity case}. By Lemma~\ref{uniformly bounded from above in prescribed singularity case} $\{u_{j}\}$ are uniformly bounded from above. Thus 
\[
v_{j} := (\sup_{k\geq j} u_{k})^{*} \in \PSH(X,\theta).
\]
Since $u_{j} \leq v_{j}$ we  know using Lemma~\ref{lem: the fundamental inequality} that $v_{j} \in \E_{\psi}(X,\theta, \phi)$. Define
\[
v_{j}^{l} := \max\{ u_{j}, u_{j+1}, \dots, u_{j+l}\}.
\]
Then $v_{j}^{l} \in \E_{\psi}(X,\theta, \phi)$ and $v_{j}^{l} \nearrow v_{j}$ almost everywhere. Lemma~\ref{lem: stability} implies $I_{\psi}(v_{j}^{l}, v_{j}) \to 0$ as $l \to \infty$. Using quasi-triangle inequality twice we get that $I_{\psi}(v_{j}, v_{j+1}) \leq C^{2} \lim_{l \to \infty}I_{\psi}(v_{j}^{l+1}, v_{j+1}^{l})$. Now, 
\begin{align*}
    I_{\psi}(v_{j}^{l+1}, v_{j+1}^{l}) &= I_{\psi}(\max\{u_{j}, v_{j+1}^{l}\}, v_{j+1}^{l}). 
    \intertext{Using Theroem~\ref{thm: pythagorean identity for max in prescribed singularity}, we get}
    &\leq I_{\psi}(u_{j}, v_{j+1}^{l}). 
    \intertext{Using Theorem~\ref{thm : quasi-triangle inequality in the prescribed singularity case}, we get}
    &\leq CI_{\psi}(u_{j}, u_{j+1})+  C I_{\psi}(u_{j+1}, \max\{u_{j+1}, v_{j+2}^{l-1}\}). 
    \intertext{This again by Theorem~\ref{thm: pythagorean identity for max in prescribed singularity} gives us}
    &\leq CI_{\psi}(u_{j},u_{j+1}) + CI_{\psi}(u_{j+1}, v_{j+2}^{l-1}).
    \intertext{Applying this $l$ times we get}
    &\leq \sum_{k=1}^{l}C^{k}I_{\psi}(u_{j+k-1}, u_{j+k}) \\
    &\leq \sum_{k=1}^{l}C^{k}C^{-2j-2k+2} \\
    &\leq \sum_{k=1}^{\infty}C^{-2j-k+2}\\
    &\leq C^{-2j+2}\frac{1}{C-1} = C^{-2j}\frac{C^{2}}{C-1}.
\end{align*}
Thus we obtain that
\[
I_{\psi}(v_{j}, v_{j+1}) \leq C^{-2j} \frac{C^{4}}{C-1}.
\]
This shows that $\{v_{j}\}$ is a decreasing $I_{\psi}$-Cauchy sequence. Thus by Lemma~\ref{lem: decreasing cauchy sequence in prescribed singularity}, we get that there exists $v \in \E_{\psi}(X,\theta,\phi)$ such that $I_{\psi}(v_{j}, v) \to 0$.  

Now we want to show that $\{u_{j}\}$ and $\{v_{j}\}$ are equivalent sequences, i.e., $I_{\psi}(u_{j},v_{j}) \to 0$. Then $I_{\psi}(u_{j},v) \leq CI_{\psi}(u_{j},v_{j}) + CI_{\psi}(v_{j},v)$. As both the terms go to 0, we get that $I_{\psi}(u_{j},v) \to 0$. 

Now, 
\begin{align*}
    I_{\psi}(u_{j}, v_{j}^{l}) &= I_{\psi}(u_{j}, \max\{ u_{j}, v_{j+1}^{l-1}\}) \\
    &\leq I_{\psi}(u_{j}, v_{j+1}^{l-1}) \\
    &\leq CI_{\psi}(u_{j}, u_{j+1}) + CI_{\psi}(u_{j+1}, v_{j+1}^{l-1}) \\
    &\leq CI_{\psi}(u_{j},u_{j+1}) + CI_{\psi}(u_{j+1}, v_{j+2}^{l-2}) \\
    &\leq CI_{\psi}(u_{j},u_{j+1}) + C^{2}I_{\psi}(u_{j+1},u_{j+2}) + C^{2}I_{\psi}(u_{j+2}, v_{j+2}^{l-2}).
    \intertext{Doing this $l$ times we get }
    &\leq CI_{\psi}(u_{j},u_{j+1}) + C^{2} I_{\psi}(u_{j+1},u_{j+2}) + \dots C^{l}I_{\psi}(u_{j+l-1}, u_{j+l}) \\
    &\leq \sum_{k=1}^{l} C^{k}C^{-2k-2j+2} \\
    &\leq C^{-2j} \frac{C^{2}}{C-1}.
\end{align*}
Thus 
\[
I_{\psi}(u_{j}, v_{j}) \leq C I_{\psi}(u_{j},v_{j}^{l}) + C I_{\psi}(v_{j}^{l},v_{j}) \leq C^{-2j}\frac{C^{3}}{C-1}
\]
by taking limit $l \to \infty$. Thus $I_{\psi}(u_{j},v_{j}) \to 0$. Hence we have found $ v \in \E_{\psi}(X,\theta,\phi)$ such that $I_{\psi}(u_{j}, v) \to 0$. Thus $(\E_{\psi}(X,\theta,\phi), I_{\psi})$ is a completely metrizable topological space when topologized with the quasi-metric $I_{\psi}$. 
\end{proof}

This finishes the proof for completeness for of the quasi-metric $I_{\psi}$ on the space $\E_{\psi}(X,\theta,\phi)$.

\section{Properties of the new topology } \label{sec: properties of the new topology} 

The following lemma generalizes  \cite[Lemma 1.5]{Guedj2019PlurisubharmonicEA} and Lemma~\ref{lem: stability}.
\begin{lem}\label{lem: decreasing sequence has finite energy in prescribed singularity case}
Let $u_{j} \in \E_{\psi}(X,\theta,\phi)$ be a decreasing sequence and $\varphi \in \E_{\psi}(X,\theta, \phi)$ be such that
\[
\sup_{j \in \N} \int_{X} \psi(u_{j} - \varphi) \theta^{n}_{u_{j}} < \infty.
\]
Then $u = \lim_{j} u_{j} \in \E_{\psi}(X,\theta,\phi)$. 
\end{lem}

\begin{proof}
First we show that $u_{j} \not \to -\infty$ uniformly. If, on the contrary, $u \to \infty$ uniformly, up to choosing a subsequence and relabeling, we can assume that $u_{j} < -j$.  Choose an $A$ such that $\theta^{n}_{\varphi}\{ \varphi > -A\} > 0$. Now
\begin{align*}
    \int_{X}\psi(u_{j} -\varphi) \theta^{n}_{u_{j}} &= \int_{0}^{\infty}\theta^{n}_{u_{j}} \{ \psi(u_{j} - \varphi) > t\} dt \\
    &= \int_{0}^{\infty} \psi'(s) \theta^{n}_{u_{j}} \{ |u_{j} - \varphi| > s\} ds \\
    &= \int_{0}^{\infty} \psi'(s) \theta^{n}_{u_{j}} \{ \varphi < u_{j} - s\} ds \\ &+ \int_{0}^{\infty} \psi'(s) \theta^{n}_{u_{j}} \{ u_{j} < \varphi - s\}ds.
    \intertext{Using comparison principle and the fact that $\psi'(s) \geq 0$, we get }
    &\geq \int_{0}^{\infty} \psi'(s) \theta_{\varphi} \{ u_{j} < \varphi - s\}ds. 
    \intertext{Since $u_{j} < -j$, we get that $\{ -j < \varphi - s\} \subset \{ u_{j} < \varphi - s\}$. Thus,}
    &\geq \int_{0}^{\infty} \psi'(s) \theta^{n}_{\varphi} \{ -j < \varphi -s \}ds. 
    \intertext{Choosing $j > A$, we can write}
    &\geq \int_{0}^{j - A} \psi'(s) \theta^{n}_{\varphi} \{ \varphi > s - j\}ds. \intertext{Again, we notice that $s < j- A$ implies that $\{ \varphi > -A \} \subset \{ \varphi > s -j \}$, which gives}
    &\geq \int_{0}^{j-A} \psi'(s)\theta^{n}_{\varphi}\{ \varphi > -A\} ds \\
    &= \theta^{n}_{\varphi} \{ \varphi > -A \} \psi(j-A) \to \infty
\end{align*}
as $j \to \infty$.
This is a contradiction to the fact that $\int_{X}\psi(u_{j} -\varphi)\theta^{n}_{u_{j}}$ is bounded in $j$. 

Moreover we observe that $I_{\psi}(u_{j}, \varphi)$ is bounded. Notice that, like in the previous argument, we have 
\begin{align*}
    \int_{X} \psi(u_{j} - \varphi) \theta^{n}_{\varphi} &= \int_{0}^{\infty} \psi'(s) \theta^{n}_{\varphi}\{ u_{j} < \varphi - s\} ds + \int_{0}^{\infty} \psi'(s) \theta^{n}_{\varphi} \{ \varphi < u_{j} - s\}. 
    \intertext{Using comparison principle for the first expression  and the fact that $u_{j} \leq u_{1}$ in the second expression, we get that }
    &\leq \int_{0}^{\infty} \psi'(s) \theta^{n}_{u_{j}} \{ u_{j} < \varphi - s\}ds + \int_{0}^{\infty } \psi'(s) \theta^{n}_{\varphi}\{ \varphi < u_{1} - s\} ds\\
    &\leq \int_{X} \psi(u_{j} - \varphi) \theta^{n}_{u_{j}} + \int_{X} \psi(u_{1} - \varphi) \theta^{n}_{\varphi}.
\end{align*}
Thus $\int_{X} \psi(u_{j} - \varphi) \theta^{n}_{\varphi}$ is bounded from above independent of $j$. 
Now notice that 
\begin{align*}
    \int_{X} \psi(u_{j} - \phi) \theta^{n}_{u_{j}} &\leq I_{\psi}(\phi,  u_{j}). 
    \intertext{By Quasi-triangle inequality for $I_{\psi}$, we get}
    &\leq C (I_{\psi}(\phi, \varphi) + I_{\psi}(u_{j}, \varphi)) \\
    &\leq C \left(I_{\psi}(\phi, \varphi) + \int_{X} \psi(u_{j} - \varphi)\theta^{n}_{u_{j}} + \int_{X} \psi(u_{j} - \varphi) \theta^{n}_{\varphi}\right).
\end{align*}
Thus $\int_{X} \psi(u_{j} - \phi) \theta^{n}_{u_{j}}$ is bounded from above independent of $j$ and $u = \lim_{j} u_{j} \in \PSH(X,\theta)$, thus by  Lemma~\ref{lem: stability} we get that $u \in \E_{\psi}(X,\theta,\phi)$.
\end{proof}

The following theorem is the generalization of \cite[Proposition 2.6]{BDL17}. Also when $\psi(t) = |t|$, the following theorem appers in \cite[Proposition 5.7]{Trusianistrongtopologyofpshfunctions}.

\begin{thm}\label{thm: sandwiching by monotone sequence in prescribed singularity case}
Let $u_{j},u \in \E_{\psi}(X,\theta,\phi)$ be such that $I_{\psi}(u_{j},u) \to 0$. Then there exits a subsequence still denoted by $u_{j}$ and $v_{j} , w_{j} \in \E_{\psi}(X,\theta,\phi)$ such that $v_{j} \leq u_{j} \leq w_{j}$ and $v_{j}$ increase to $u$ a.e. and $w_{j}$ decrease to $u$. Thus, by Proposition~\ref{convergence with monotone sequences in prescribed singularity setting} and monotone convergence theorem, $I_{\psi}(v_{j}, u) \to 0$ and $I_{\psi}(w_{j},u) \to 0$. 
\end{thm}

\begin{proof}
We can pass to a subsequence of $(u_{j})$ such that $I_{\psi}(u_{j}, u) \leq C^{-2j}$ where $C$ is the same constant as in Theorem~\ref{thm : quasi-triangle inequality in the prescribed singularity case}.  By quasi-triangle inequality, $I_{\psi}(u_{j}, u_{j+1}) \leq C^{-2j+2}$. By Lemma~\ref{uniformly bounded from above in prescribed singularity case}, $u_{j}$ are uniformly bounded from above. Thus
\[
w_{j} := (\sup_{k\geq j} u_{k})^{*} \in \E_{\psi}(X,\theta,\phi).
\]
Moreover, by the proof of Theorem~\ref{thm: completeness in prescribed singularity setting} we get that $w_{j}$ is a decreasing sequence in $\E_{\psi}(X,\theta)$ and is equivalent to $u_{j}$. Hence $w_{j} \geq u_{j}$ and $w_{j} \searrow u$ and thus $I_{\psi}(w_{j}, u) \to 0$. 

For $j < k$, define $v_{j}^{k} = P_{\theta}(\min(u_{j} , \dots, u_{k}))$. By Theorem~\ref{thm: operator P is stable on E psi}, $v_{j}^{k} \in \E_{\psi}(X,\theta,\phi)$. Moreover, by \cite[Lemma 3.7]{Darvasmonotonicity}, 
\[
\theta^{n}_{v^{k}_{j}} \leq \sum_{l=j}^{k}\mathds{1}_{\{v_{j}^{k} = u_{l} \}}\theta^{n}_{u_{l}}.
\]
Therefore 
\[
\int_{X}\psi(u - v_{j}^{k}) \theta^{n}_{v_{j}^{k}} \leq \sum_{l=j}^{k} \int_{X}\psi(u - u_{l}) \theta^{n}_{u_{l}} \leq \sum_{l = j}^{k} I_{\psi}(u,u_{l})\leq C^{-2j+2}.
\]
Also $v_{j}^{k}$ is decreasing as $k \to \infty$, thus $v_{j} := \lim_{k}v_{j}^{k} \in \E_{\psi}(X,\theta,\phi)$ by Lemma~\ref{lem: decreasing sequence has finite energy in prescribed singularity case} $v_{j} \in \E_{\psi}(X,\theta,\phi)$. 

Since $v_{j}^{k}$ decreases to $v_{j}$, we get that $v_{j}^{k}\to v_{j}$ in capacity. Moreover, the functions $\psi(u-v_{j}^{k}) \to \psi(u - v_{j})$ in capacity.  Using Lemma~\ref{thm: removing uniform boundedness assumption}, we get that 
\[
\int_{X} \psi(u - v_{j})\theta^{n}_{v_{j}} \leq \liminf_{k\to \infty} \int_{X}  \psi( u - v_{j}^{k}) \theta^{n}_{v_{j}^{k}} \leq C^{-2j+2}.
\]

Since $(\sup_{k\geq j} u_{k})^{*} = w_{j}$, therefore, $\sup_{k\geq j} u_{k} = w_{j}$  a.e. As $w_{j} \searrow u$, we get that $\limsup_{k}u_{k} = u$ a.e. For any $v_{j}$, we have $v_{j} \leq u_{k}$ for $k \geq j$. Taking limsup we get $v_{j} \leq \limsup_{k\geq j} u_{k} = u$. Therefore, $v = (\lim_{j} v_{j})^{*} \leq u$.

By the same argument as before, we have,
\[
\int_{X} \psi(u-v) \theta^{n}_{v} \leq \liminf_{j \to \infty} \int_{X}\psi(u - v_{j})\theta^{n}_{v_{j}} \leq \liminf_{j \to \infty} C^{-2j+2} = 0.
\]
Thus $\theta^{n}_{v}(\{u \neq v\} = 0$. Again using the domination principle (see Lemma~\ref{lem: domination principle}), we get that $u \leq v$ everywhere. This shows that $u = v$. 

Thus we found an increasing sequence $v_{j} \leq u_{j}$ increasing to $u$.
\end{proof}

\begin{cor}\label{cor: new topology is stronger in prescribed singularity case}
The $I_{\psi}$-topology on $\E_{\psi}(X,\theta,\phi)$ is stronger than the usual $L^{1}(\om^{n})$ topology on $\E_{\psi}(X,\theta, \phi)$. More precisely, if $u_{k},u \in \E_{\psi}(X,\theta,\phi)$ such that $I_{\psi}(u_{k},u) \to 0$, then 
\[
\int_{X} |u_{k} - u | \om^{n} \to 0
\]
as $k \to \infty$. 
\end{cor}

\begin{proof}
It is enough to show $L^{1}(\om^{n})$ convergence for a subsequence. Let $u_{j_{k}}$ be the subsequence provided by Theorem~\ref{thm: sandwiching by monotone sequence in prescribed singularity case} and $v_{j_{k}}$ and $w_{j_{k}}$ be corresponding monotone sequences. Then $v_{j_{k}} \leq u_{j_{k}} \leq w_{j_{k}}$ and $v_{j_{k}} \leq u \leq w_{j_{k}}$. Then 
\begin{align*}
\int_{X} |u_{j_{k}} - u| \om^{n} &\leq \int_{X} (w_{j_{k}} - v_{j_{k}})\om^{n} \\
&\leq \int_{X} (w_{j_{k}} - u) \om^{n} + \int_{X} (u - v_{j_{k}})\om^{n}\\
& \to 0
\end{align*}
by the monotone convergence theorem. Thus the new topology is stronger and has more open, thus closed sets. 
\end{proof}

\begin{thm}\label{thm: convergence in capacity in prescribed singularity case}
If $u_{j}, u \in \E_{\psi}(X,\theta,\phi)$, such that $I_{\psi}(u_{j}, u ) \to 0$ as $j \to \infty$, then $u_{j} \to u$ in capacity.
\end{thm}

\begin{proof}
It is enough to show the convergence in capacity for a subsequence. Let $u_{j_{k}}$ be a subsequence as provided by Theorem~\ref{thm: sandwiching by monotone sequence in prescribed singularity case} and $v_{j_{k}}$ and $w_{j_{k}}$ are corresponding monotone sequences converging to $u$.  We get that $v_{j_{k}} \to u$ and $w_{j_{k}} \to u$ in capacity. We want to claim that $u_{j_{k}} \to u$ in capacity as well. Fix $\ee > 0$. 
\[
\{ | u_{j_{k}} - u| > \ee\} \subset \{ w_{j_{k}} - v_{j_{k}} > \ee\} \subset \{ w_{j_{k}} - u > \ee/2\} \cup \{ u - v_{j_{k}} > \ee/2\}.
\]
Taking capacity of above sets we get
\[
\CAP_{\om}\{|u_{j_{k}} - u|> \ee\} \leq \CAP_{\om}\{ w_{j_{k}} - u> \ee/2\} + \CAP_{\om}\{ u - v_{j_{k}} > \ee/2\}
\]
and taking limit $k \to \infty$ we get $\lim_{k\to \infty}\CAP_{\om}\{ |u_{j_{k}} - u| > \ee\} = 0$. 

Since for any subsequence $u_{j_{k}}$ of $u_{j}$ we can find another subsequence $u_{j_{k_{l}}}$ for which $\lim_{l\to \infty}\CAP_{\om}\{|u_{j_{k_{l}}} - u| > \ee\} = 0$, we get that $\lim_{j\to \infty}\CAP_{\om}\{|u_{j} - u|>\ee\} = 0$. Thus $I_{\psi}$ convergence implies convergence in capacity.
\end{proof}

\begin{cor}[Weak convergence of measures] \label{cor: weak convergence of measures in prescribed singularity case}
If $u_{k}, u \in \E_{\psi}(X,\theta, \phi)$ such that $I_{\psi}(u_{k}, u ) \to 0$, then $\theta^{n}_{u_{k}} \to \theta^{n}_{u}$ weakly as measures. 

\end{cor}

\begin{proof}
Since $u_{k} \to u$ in $I_{\psi}$, Theorem~\ref{thm: convergence in capacity in prescribed singularity case} shows that $u_{k} \to u$ in capacity. Since $u_{k}, u$ have full mass,  \cite[Theorem 2.3]{Darvasmonotonicity} (see also \cite[Theorem 1]{XingCMAonopendomainsr} and  \cite[Theorem 1]{XingCMAoncompactKahler}) implies that $\theta^{n}_{u_{k}} \to \theta^{n}_{u}$ weakly.
\end{proof}

\section{K\"ahler Ricci flow} \label{sec: Kahler Ricci flow}
In \cite{Guedj2013RegularizingPO}, authors showed that we can start K\"ahler Ricci flow from any potential $\varphi$ with zero Lelong numbers. They showed that if $\varphi \in \PSH(X,\om)$ has zero Lelong numbers then there exist smooth potentials $\varphi_{t}$ for small time such that 
\begin{equation} \label{eq: kahler ricci flow equation}
    \frac{\dd\varphi_{t}}{\dd t} = \log \left[ \frac{(\om +dd^{c}\varphi_{t})^{n}}{\om^{n}} \right], \varphi_{t} \to \varphi \text{ in } L^{1}(\om^{n}) \text{ as } t\to 0.
\end{equation}

Di Nezza-Lu \cite[Corollary 5.2]{dinezzalukahlerricciflow} further showed that $\varphi_{j} \to \varphi$ in capacity. 

 Without loss of generality, we can assume that $\varphi \leq -1$.  In \cite[Lemma 2.9]{Guedj2013RegularizingPO}, the authors showed that (for $\bb =1$ and $\al = 1$) for a continuous $\om$-psh function $u$, satisfying 
\[
(\om + dd^{c}u)^{n} = e^{u - 2\varphi}\om^{n}
\]
we have
\[
(1-2t)\varphi (x) + tu(x) + n(t\ln t - t) \leq \varphi_{t}.
\]

Since $\varphi \leq -1$, we have $\varphi \leq (1-2t)\varphi$ for small $t$ and since $u$ is continuous there exists some constant $C$ such that $C \leq u$. Combining these, we get 
\[
\varphi(x) + tC + n(t\ln t - t) \leq \varphi_{t}
\]
or 
\[
\varphi(x) \leq \varphi_{t} - tC - n(t\ln t - t).
\]
Notice that $-tC - n(t\ln t - t) = f(t)$ satisfies $f(t) \to 0$ as $t \to 0$. 

Since $\varphi_{t} \to \varphi$ in $L^{1}(\om^{n})$ as $t\to 0$, we also obtain that $\varphi_{t} + f(t)$ converge to $\varphi$ in $L^{1}(\om^{n})$ as $t\to 0$.

\begin{lem} \label{lem: l1 convergence implies convergence of measures}
Given $\varphi \in \E_{\psi}(X,\om)$ and a sequence $\varphi_{j} \in \E_{\psi}(X,\om)$ such that 
\[
\varphi \leq \varphi_{j}
\]
and $\varphi_{j} \to \varphi$ in $L^{1}(\om^{n})$ as $j \to \infty$, then $I_{\psi}(\varphi_{j}, \varphi) \to 0$ as $j \to \infty$.. 
\end{lem}

\begin{proof}

Since $\varphi_{j} \to \varphi$ in $L^{1}(\om^{n})$, we get a subsequence $\varphi_{j_{k}}$ such that $\varphi_{j_{k}} \to \varphi$, a.e.. Consider the functions 
\[
v_{j} = \sup_{k\geq j} \varphi_{j_{k}}.
\]
Then $v_{j}^{*} \in \PSH(X,\om)$ and $v_{j}^{*} = v_{j}$ except on a pluripolar set. Moreover, since $v_{j}^{*} \searrow \varphi$ $\om^{n}$ almost everywhere, we get that $v_{j}^{*} \searrow \varphi$ everywhere. 

This shows that $v_{j} \searrow \varphi$ except on a pluripolar set. This means $\lim_{j\to \infty} \sup_{k\geq j}\varphi_{k_{j}} = \varphi$ except on a pluripolar set. Therefore
\[
\limsup_{k\to \infty} \varphi_{j_{k}} = \varphi 
\]
except on a pluripolar set. Since $\varphi_{j_{k}} \geq \varphi$, we automatically have 
\[
\liminf_{k\to \infty} \varphi_{j_{k}} \geq \varphi
\]
everywhere. 

Therefore, $\varphi_{j_{k}} \to \varphi$ except on a pluripolar set. Using $\varphi \leq \varphi_{j_{k}} \leq v_{k}$ and that $v_{k}\searrow \varphi$ as $k \to \infty$, we get 
\begin{equation}
\lim_{k\to \infty} \int_{X}\psi(\varphi_{j_{k}} - \varphi) \om^{n}_{\varphi} \leq \lim_{k \to \infty} \int_{X} \psi(v_{k} - \varphi) \om^{n}_{\varphi} = 0. 
\end{equation}
Moreover, since $\varphi_{j_{k}} \geq \varphi$, we have
\begin{align*}
    \int_{X}\psi(\varphi_{j_{k}} - \varphi) \om^{n}_{\varphi_{j_{k}}} &= \int_{0}^{\infty} \psi'(s) \om^{n}_{\varphi_{j_{k}}} (\varphi_{j_{k}} > \varphi + s)ds \intertext{Using comparison theorem, we get}
    &\leq \int_{0}^{\infty} \psi'(s) \om^{n}_{\varphi}(\varphi_{j_{k}} > \varphi + s) ds \\
    &= \int_{X} \psi(\varphi_{j_{k}} - \varphi)\om^{n}_{\varphi}. 
    \intertext{Taking limit we get }
    \lim_{k\to \infty} \int_{X}\psi(\varphi_{j_{k}} - \varphi)\om^{n}_{\varphi_{j_{k}}} &\leq \lim_{k\to \infty} \int_{X} \psi(\varphi_{j_{k}} - \varphi)\om^{n}_{\varphi} = 0.
\end{align*}
Combining these two we get 
\[
\lim_{k \to \infty} I_{\psi}(\varphi_{j_{k}}, \varphi) = \lim_{k\to \infty} \int_{X} \psi(\varphi_{j_{k}} - \varphi)(\om^{n}_{\varphi_{j_{k}}} + \om^{n}_{\varphi}) = 0.
\]
Since this holds for any subsequence of $\varphi_{j}$, we get that 
\[
I_{\psi}(\varphi_{j}, \varphi) \to 0.
\]
\end{proof}

The following corollary generalizes the \cite[Proposition 5.2]{Guedj2013RegularizingPO} where they show that $\varphi_{t} \to \varphi$ in energy if $\varphi \in \E^{1}(X,\om)$.
\begin{cor}\label{cor: measures in kahler ricci flow}
Let $\varphi \in \E_{\psi}(X,\om)$. Let $\varphi_{t}$ be the solution to Equation~\eqref{eq: kahler ricci flow equation}. Then $I_{\psi}(\varphi_{t}, \varphi) \to 0$ as $t \to 0$. 
\end{cor}

\begin{proof}
We saw earlier that the assumptions imply  that
\[
\varphi \leq \varphi_{t} + f(t)
\]
where $f(t) \to 0$ as $t \to 0$. Since $\varphi_{t} + f(t)$ are also $\theta$-psh functions which converge in $L^{1}(\om^{n})$ to $\varphi$ as $t\to 0$, using Lemma~\ref{lem: l1 convergence implies convergence of measures}, we get that $I_{\psi}(\varphi_{t}, \varphi) \to 0$ as $t \to 0$. 
\end{proof}

Note that using Theorem~\ref{thm: convergence in capacity in prescribed singularity case} and Corollary~\ref{cor: weak convergence of measures in prescribed singularity case} we obtain that $\varphi_{t} \to \varphi$ in capacity and $\om^{n}_{\varphi_{t}} \to \om^{n}_{\varphi}$ as $t \to 0$, thus finishing the proof of Theorem~\ref{thm: kahler-ricci flow}.

\printbibliography

\end{document}